\documentclass[reqno,9pt]{amsart}
\usepackage{wrapfig}
\usepackage[hang,small,bf]{caption}
\usepackage[subrefformat=parens]{subcaption}
\usepackage{comment}
\usepackage{indentfirst}
\usepackage[top=20mm, bottom=20mm, left=30mm, right=30mm]{geometry}
\usepackage{amscd}
\usepackage{amsmath,amsthm,amssymb}
\usepackage{color}
\usepackage[dvipdfmx]{graphicx}
\usepackage{amsfonts}
\usepackage{cases}
\usepackage{bm}
\usepackage{pb-diagram}
\usepackage{braket}
\usepackage{multirow}
\usepackage{empheq}
\usepackage{tikz}

\usepackage{listings,jvlisting}

\usetikzlibrary{intersections,calc,arrows}

\numberwithin{equation}{subsection}
\theoremstyle{definition}
\newtheorem{defi}{Definition}[section]

\newtheorem{lem}[defi]{Lemma}

\newtheorem{theo}[defi]{Theorem}

\newtheorem{remark}[defi]{Remark}
\newtheorem{conj}[defi]{Conjecture}

\newcommand{\initial}{\mathrm{in}}

\newcommand{\join}{\vee}
\newcommand{\meet}{\wedge}

\title{
New examples of radical join-meet ideals
}

\author{Yohei Oshida}

\address{Graduate School of Engineering and Science , Shibaura Institute of Technology, 307 Minumaku Fukasaku, Saitama-City 337-8570, Japan.}

\email{mf20022@shibaura-it.ac.jp}


\begin{document}

\begin{abstract}
We give new examples of a finite lattice $L$ such that the join-meet ideal $I_{L}$ is radical. 
\end{abstract}

\maketitle{}

\setcounter{tocdepth}{1}
\tableofcontents

\section{Introduction}
Let $L$ be a finite lattice and $K[L]$ be the polynomial ring over a field $K$ whose variables are the elements of $L$. The ideal
\begin{eqnarray}
I_L=(\{ f_{a,b}:=x_a x_b - x_{a \join b} x_{a \meet b} \mid a , b \in L \}) \subset K[L] \nonumber
\end{eqnarray}
is called the join-meet ideal of $L$. It was introduced in 1987 by Hibi in \cite{3}. As shown by \cite{1} or \cite{3}, $L$ is distributive if and only if $I_{L}$ is a prime ideal. It follows from this result that $I_L$ is radical when $L$ is distributive. However, it is not yet completely known classes of non-distributive lattice L with the property that $I_L$ is a radical ideal. On the other hand, for instance, it followed from \cite{1} and \cite{2} that there are some examples of non-distributive modular lattice such that $I_L$ is a radical ideal.

Not all, here we briefly introduce three examples. First, The join-meet ideal of the pentagon lattice $N_5$ and diamond lattice $M_5$ is radical; see [1, Page 157] for detail. Second, for some integer $n \geq 1$, it exists a class of the distributive lattice of the divisors of  $2 \cdot 3^n$ such that by including just one small diamond one get a radical join-meet ideal for the new lattice; see [2, Section 3] for detail.

In this paper, we introduce two new examples of non-distributive lattices $L$ such that the join-meet ideal $I_{L}$ is radical. Let $k$ be non-negative integer with $k >0$. For non-negative integer $n_1,\cdots,n_k \geq 1$, let 
\begin{eqnarray}
L_k(n_1,\cdots,n_k)=\{ s , a_{1,1} , \cdots , a_{1,n_1} , a_{2,1} , \cdots , a_{2,n_2} , \cdots , a_{k,1} , \cdots, a_{k,n_k} , t \} \nonumber 
\end{eqnarray}
be a finite lattice with $s < a_{i,1} < a_{i,2} < \cdots < a_{i,n_i} < t$ for $i=1,\cdots,k$. In terms of appearance of the Hasse diagram, A finite lattice $L_k(n_1,\cdots,n_k)$ looks like stretched vertically and horizontally of a finite lattice $\{ \xi , a , b , c , d , \zeta \}$ introduced in [1, Problems 6.13]. Then, the following question arises. Is the join-meet ideal $I_{L_k(n_1,\cdots,n_k)}$ radical ? Unfortunately, I couldn't answer this question. On the other hand, for $ k = 2,3 $, we obtained the following theorem.

\begin{theo}\it \label{theo1}The join-meet ideal $I_{L_2(n_1,n_2)}$ is radical.
\end{theo}
\begin{theo}\it \label{theo2}For
\begin{eqnarray}
(n_1,n_2,n_3)&=&(k_1,1,1), \quad 2 \leq k_1 \leq 10, \nonumber \\
&&(k_2,k_3,1), \quad 2 \leq k_2 \leq 4,2 \leq k_3 \leq 4, \nonumber \\
&&(k_4,k_5,k_6), \quad 2 \leq k_4 , k_5 , k_6 \leq 3, \nonumber
\end{eqnarray}
the join-meet idea $I_{L_3(n_1,n_2,n_3)}$ is radical.
\end{theo}
By using new examples $L_2(n_1,n_2)$, we obtained new non-distributive non-modular lattices $L_2(n_1,n_2)[k',i_1,i_2]$ for non-negative integer $k',n_1,n_2$ satisfying certain conditions. We also obtained new distribtuive lattices  $O_{n_1}$. Then, we obtained the following theorem.

\begin{theo}\label{theo4}\it
The join-meet ideal $I_{L_2(n_1,n_2)[{k'},i_1,i_2]}$ is radical.
\end{theo}

\begin{theo}\label{theo3}\text{\it The join-meet ideal $I_{O_{n}}$ is radical.}
\end{theo}

We checked that  Theorem \ref{theo4} and \ref{theo3} are similar to [2, Theorem 3.3] in terms of the opposite approach. Unfortunately, since $O_{n_1}$ is a distributive lattice, note that it not new exmaples. The detailed definition of $L_2(n_1,n_2)[k',i_1,i_2]$ and $O_{n_1}$ is given in section 4.

This paper is organized as follows. In section 2, we introduce the proof of Theorem \ref{theo1} and \ref{theo2}. In section 3, we introduce the conjectures that occur naturally by Theorem \ref{theo1} and $\ref{theo2}$. Then, we give some thoughts. In section 4, we introduce the proof of Theorem \ref{theo4} and \ref{theo3}. In section 5, we introduce topics related to $O_ {n_1}$. The keywords of it are number theory and gorenstein ring. Note that it has little to do with the gist of this paper.\\

Below, unless otherwise noted, in order to avoid the complexity of notation, we denote 
\begin{eqnarray}
a_{1,i}=a_i \quad \text{for $1 \leq i \leq n_1$}, \quad a_{2,i}=b_i \quad \text{for $1 \leq i \leq n_2$}, \quad a_{3,i}=c_i \quad \text{for $1 \leq i \leq n_3$}. \nonumber
\end{eqnarray}
Furthermore, in order to match the logical calculation, let $a_i$, $b_i$ and $c_i$ satisfy
\begin{eqnarray}
&&a_i,b_i,c_i=s \quad \text{for $i \leq 0$}, \nonumber \\
&&a_i=t \quad \text{for $n_1+1 \leq i$}, \quad b_i=t \quad \text{for $n_2+1 \leq i$}, \quad c_i=t \quad \text{for $n_3+1 \leq i$}. \nonumber
\end{eqnarray}

\section{The proof of theorem \ref{theo1} and theorem \ref{theo2}}

In this section, we introduce the proof of theorem \ref{theo1} and theorem \ref{theo2}. Hereafter, in order to avoid the complexity of notation, let $n_1 =n$, $n_2=m$, $n_3=r$.

\subsection{The proof of Theorem \ref{theo1}}

Let
\begin{eqnarray}
G(n,m)&=&\{ f_{a,b}=ab-st \mid a , b \in L_2(n,m) \}, \nonumber \\
A_n&=&\{ a_i s t - a_1 s t \mid 1 < i \leq n \}, \nonumber \\
B_m&=&
\begin{cases}
\emptyset, & m=1, \\
\{ b_i s t - b_1 s t \mid 1 < i \leq m \}, & m>1.
\end{cases}
\nonumber
\end{eqnarray}

The outline of the proof of Theorem \ref{theo1} is to show that $I_{L_2(n,m)}$ is squarefree with respect to the inverse lexicographic order induced by
\begin{eqnarray}\label{prime order}
s \prec a_1 \prec \cdots \prec a_n \prec b_1 \prec \cdots \prec b_m \prec c_1 \prec \cdots \prec c_r \prec t.
\end{eqnarray}
To prove this claim, for $n>1$, we show that the set $G_{n,m} \cup A_{n} \cup B_{m}$ is a Gr\"{o}bner basis of $I_{L_2(n,m)}$ with repsect to $\prec$. Below, by using \text{\bf{Buchberger's criterion}}, we show each case when $m= 1$ and when it is not. Note that it is clearly that $I_{L_ 2(1,1)}$ is radical; see [1, Theorem 6.10(Dedekind)] and [1, Theorem 6.21].

\subsubsection{The case $m=1$}

First, for $u$ and $v$ belonging to $G_{n,1}$, we show that the $S$-polynomial $S(u,v)$ reduces to $0$. Let $i$ and $j$ be non-negative integer with $1 \leq i,j \leq n $. Let $u_{i,j}$ denote the $S$-polynomial $S(a_i b_1 - s t , a_j b_1 - s t)$. If $i=j$, then we have $u_{i,j}=u_{i,i}=0$. On the other hand, if $i \neq j$, we have
\begin{eqnarray}\label{m,1,first}
u_{i,j}=a_j ( a_i b_1 - s t ) - a_i ( a_j b_1 - s t )=- a_j s t + a_i s t.
\end{eqnarray}
Thus, computational result of $u_ {i, j} (i \neq j)$ is as Table 1. Hence, we showed that $S(u,v)$ reduces to $0$.\\
\begin{table}[h]
 \caption{Computational result of $u_ {i, j}(i \neq j)$}
 \label{table1}
 \centering
  \begin{tabular}{|c|c|c|c|c|}
   \hline
   Value of $i$ & Value of $j$ & A standard expression of $u_{i,j}(i \neq j)$ \\
   \hline \hline
   $i=1$ & $j=1$ & 0 \\
   \cline{2-3}
   & $j > 1$ & $- (a_j s t - a_1 s t)$ \\
   \hline
   $i > 1$ & $j=1$ & $a_i s t - a_1 s t$ \\
   \cline{2-3}
   & $j > 1$ & $(a_i s t - a_1 s t ) - (a_j s t - a_1 s t)$ \\
   \hline
  \end{tabular}
\end{table}

Second, for $u$ and $v$ belonging to $\in A_{n}$, we show that $S(u,v)$ reduces to $0$. Let $i$ and $j$ be non-negative integer with $2 \leq i,j \leq n $. Let $u_{i,j}$ denote the $S$-polynomial $S(a_i s t - a_1 s t , a_j s t - a_1 s t)$. If $i=j$, then we have $u_{i,j}=u_{i,i}=0$. On the other hand, if $i \neq j$, then we have 
\begin{eqnarray}
u_{i,j}=a_j ( a_i s t - a_1 s t ) - a_i ( a_j s t - a_1 s t )=- a_j a_1 s t + a_i a_1 s t=a_1 (a_i s t - a_1 s t) - a_1 (a_j s t - a_1 s t). \nonumber
\end{eqnarray}
Hence, we showed that the $S(u,v)$ reduces to $0$.\\

Finally, for $(u,v)$ belonging to $G_{n,1} \times A_n$, we show that $S(u,v)$ reduces to $0$. Let $i$ and $j$ be non-negative integer with $1 \leq i \leq n $, $1 < j \leq n$. Let $u_{i,j}$ denote the $S$-polynomial $S(a_i b_1 - s t , a_j s t - a_1 s t)$. If $i=j$, then we have
\begin{eqnarray}
u_{i,i}=s t ( a_i b_1 - s t ) - b_1 ( a_i s t - a_1 s t )=a_1 b_1 st - s^2 t^2=st (a_1 b_1 - st). \nonumber
\end{eqnarray}
On the ohter hand, if $i \neq j$, then we have
\begin{eqnarray}\label{n,1,0,0,ag2}
u_{i,j}&=&a_j s t ( a_i b_1 - s t ) - a_i b_1 ( a_j s t - a_1 s t ) \nonumber \\
&=&- a_j s^2 t^2 + a_i a_1 b_1 st \nonumber \\
&=&a_1 st (a_i b_1 - s t) - s t (a_j s t - a_1 s t). \nonumber
\end{eqnarray}
Hence, we showed that $S(u,v)$ reduces to $0$.\\

Therefore, We showed that the set $G_{n,m} \cup A_{n} \cup B_{1}$ is a Gr\"{o}bner basis of $I_{L_2(n,1)}$ with repsect to $\prec$.\\

\subsubsection{The case $m>1$}

First, for $u$ and $v$ belonging $G_{n,m}$, we show that the $S$-polynomial $S(u,v)$ reduces to $0$. Let $i$, $j$, $k$ and $r$ be non-negative integer with $1 \leq i \leq n$, $1 \leq j \leq m$, $1 \leq k, r \leq m$. Let $u_{i,j,k,r}$ denote the $S$-polynomial $S(a_i b_j - s t , a_k b_r - s t)$. If $i=k$, then we have $u_{i,j,i,r}=- b_r st + b_j st$. Thus, computational result of $u_ {i,j,i,r}$ from table 2. Hence, $ u_ {i,j,i,r}$ reduces to $0$.\\

\begin{table}[h]
 \caption{Computational result of $u_ {i,j,i,r}$}
 \label{table2}
 \centering
  \begin{tabular}{|c|c|c|c|c|}
   \hline
   Value of $j$ & Value of $r$ & A standard expression of $u_ {i,j,i,r}$ \\
   \hline \hline
   $j = 1$ & $r=1$ & $0$ \\
   \cline{2-3}
   & $r > 1$ & $-(b_r s t - b_1 s t)$ \\
   \hline
   $j > 1$ & $r=1$ & $b_j s t - b_1 s t$ \\
   \cline{2-3}
   & $r > 1$ & $(b_j s t - b_1 s t) - (b_r s t - b_1 s t)$ \\
   \hline
  \end{tabular}
\end{table}
On the other hand, for $j=r$, it follows that $ u_ {i,j,i,r}$ reduces to $0$  by rewriting $b_1$ to $b_j$ in $(\ref{m,1,first})$  and using table $\ref{table1}$. Therefore, we showed that $S(u,v)$ reduces to $0$.\\

Second, for $u$ and $v$ belonging to $A_n \cup B_m$, we show that $S(u,v)$ reduces to $0$. Since $S(u,v) (u,v \in A_n)$ and $S(u,v) (u,v \in B_m)$ reduce $to$ 0 from the discussion in case $m=1$, it suffices to prove that $S(u,v)$ reduces to $0$, where $(u,v) \in A_n \times B_m$.\\

Now, let $i$ and $j$ be non-negative integer with $1 < i \leq n ,1 < j \leq m$. Let $u_{i,j}$ denote the $S$-polynomial $S(a_i st - a_1 st , b_j st - b_1 st)$. Then, the polynomial $u_{i,j}$ is computed as follows:
\begin{eqnarray}
u_{i,j}=b_j ( a_i s t - a_1 s t ) - a_i ( b_j s t - b_1 s t )=- a_1 b_j s t + a_i b_1 s t=s t (a_i b_1 - s t) - s t (a_1 b_j - s t).  \nonumber
\end{eqnarray}
Hence, we showed that $S(u,v)$ reduces to $0$.\\

Finally, for $(u,v)$ belonging to $G_{n,m} \times A_n \cup B_m$, we show that $S(u,v)$ reduces to $0$. Let $i$, $j$, $k$ and $r$ be non-negative integer with $1 \leq i \leq n$, $1 \leq j, \leq m$, $1 < k \leq n$, $1 < r \leq m$. Let $u_{i,j,k}$ be the $S$-polynomial $S(a_i b_j - st , a_k st - a_1 st)$ and  $u_{i,j,r}$ the $S$-polynomial $S(a_i b_j - st , b_k st - b_1 st)$.\\

At first, about computational result of $u_{i,j,k}$, if $k=i>1$, then we have
\begin{eqnarray}
u_{i,j,k}=s t (a_i b_j - s t )- b_j (a_i s t - a_1 s  t)=a_1 b_j s t - s^2 t^2=s t (a_1 b_j - s t). \nonumber 
\end{eqnarray}
Hence, $u_{i,j,i}$ reduces to $0$. On the other hand, if $i \neq k$, then it follows that $\initial_{\prec}(a_i b_j - s t)=a_i b_j$ and $\initial_{\prec}(a_k s t - a_1 s  t)=a_k s t$ are relatively prime. Hence, for $i \neq k$, $u_{i,j,k}$ reduces to $0$ with respect to $a_i b_j - s t$, $a_k s t - a_1 s  t$.\\

Next, about the computational result of $u_{i,j,r}$, if $r=j>1$, then we have
\begin{eqnarray}
u_{i,j,k}=s t (a_i b_j - s t )- a_i (b_j s t - b_1 s  t)=a_i b_1 s t - s^2 t^2=s t (a_i b_1 - s t). \nonumber 
\end{eqnarray}
Hence, $u_{i,j,j}$ reduces to $0$. On the other hand, if $j \neq r$, it follows that $\initial_{\prec}(a_i b_j - s t)=a_i b_j$ and $\initial_{\prec}(b_r s t - b_1 s  t)=b_r s t$ are relatively prime. Hence, for $j \neq r$, $u_{i,j,r}$ reduces to $0$ with respect to $a_i b_j - s t$, $b_r s t - b_1 s  t$.\\

From the discussion of computational result of $u_{i,j,k}$ and $u_{i,j,r}$, we showed that $S(u,v)$ reduces to $0$.\\

Therefore, We showed that the set $G_{n,m} \cup A_{n} \cup B_{m}$
is a Gr\"{o}bner basis of $I_{L_2(n,m)}$ with repsect to $\prec$.

\subsubsection{Conclusion}

The set $G_{n,m} \cup A_{n} \cup B_{m}$ is a Gr\"{o}bner basis of $I_{L_2(n,m)}$ with repsect to $\prec$. Thus, we have
\begin{eqnarray}
\initial_{\prec}(I_{L_2(n,m)})=(\{ a_i b_j \mid 1 \leq i \leq n , 1 \leq j \leq m \} \cup \{ a_i s t \mid 1 \leq i \leq n \} \cup \{ b_i st \mid 1 \leq i \leq m \}). \nonumber
\end{eqnarray}
Hence, $\initial_{<}(I_{L_2(n,m)})$ is squarefree with respect to $\prec$. Therefore, $I_{L_2(n,m)}$ is radical.

\subsection{The proof of Theorem \ref{theo2}}

Let denote the following ideals:
\begin{eqnarray}
E_{n,m,r}&=&
(a_1 - a_n , \cdots , a_1 - a_2 , a_1 - b_m, \cdots , a_1 - b_1 , a_1 - c_r , \cdots , a_1-c_1 , s t - a^2_1), \nonumber \\
X_{n,m}&=&(s,a_1,\cdots,a_n,b_1,\cdots,b_m) \cap (a_1,\cdots,a_m,b_1,\cdots,b_m,t), \nonumber \\
Y_{m,r}&=&(s,b_1,\cdots,b_m,c_1,\cdots,c_r) \cap (b_1,\cdots,b_m,c_1,\cdots,c_r,t), \nonumber \\
Z_{n,r}&=&(s,a_1,\cdots,a_n,c_1,\cdots,c_r) \cap (a_1,\cdots,a_n,c_1,\cdots,c_r,t). \nonumber
\end{eqnarray}

The outline of proof is to show that all primary ideals appearing in  the primary decomposition of $I_{L_3(n,m,r)}$ is prime ideal.\\

First, by using Risa/Asir \cite{5} , we have
\begin{eqnarray}\label{equation1}
I_{L_3(n,1,1)}&=&E_{n,1} \cap X_{n,1} \cap Y_{1,1} \cap Z_{n,1} \quad \text{for $n=2,3,\cdots,10$}, \\
\label{equation2}
I_{L_3(n,1,1)}&=&E_{n,m} \cap X_{n,m} \cap Y_{m,1} \cap Z_{n,1} \quad \text{for $2 \leq n \leq 4$, $2 \leq m \leq 4$}, \\
\label{equation3}
I_{L_3(n,m,r)}&=&E_{n,m,r} \cap X_{n,m} \cap Y_{m,r} \cap Z_{n,r} \quad \text{for $2 \leq n , m , r \leq 3$}.
\end{eqnarray}
We comment a little here. The above results were obtained by doing something like the following computation of the case $(n,m,r)=(3,1,1)$:
\begin{verbatim}
[1] load("primdec")$
[2] primadec([a_1 * b_1 - s * t , a_1 * c_1 - s * t ,
a_2 * b_1 - s * t , a_2 * c_1 - s * t ,
a_3 * b_1 - s * t , a_3 * c_1 - s * t , b_1 * c_1 - s * t],[t,c_1,b_1,a_3,a_2,a_1,s]);
[3] [[[a_1-a_2,a_1-a_3,a_1-b_1,a_1-c_1,t*s-a_1^2],[a_1-a_2,a_1-a_3,a_1-b_1,a_1-c_1,t*s-a_1^2]],
[[a_1,a_2,a_3,b_1,t],[a_1,a_2,a_3,b_1,t]],[[s,a_1,a_2,a_3,b_1],[s,a_1,a_2,a_3,b_1]],
[[a_1,a_2,a_3,c_1,t],[a_1,a_2,a_3,c_1,t]],[[s,a_1,a_2,a_3,c_1],[s,a_1,a_2,a_3,c_1]],
[[b_1,c_1,t],[b_1,c_1,t]],[[s,b_1,c_1],[s,b_1,c_1]]]
\end{verbatim}\vspace{0.1in}

By $(\ref{equation1})$, $(\ref{equation2})$, $(\ref{equation3})$, we have
\begin{eqnarray}
\sqrt{I_{L_3(n,1,1)}}&=&\sqrt{E_{n,1,1}} \cap \sqrt{X_{n,1} \cap Y_{1,1} \cap Z_{n,1}} \quad \text{for $n=2,3,\cdots,10$}, \nonumber \\
\label{equation22}
\sqrt{I_{L_3(n,1,1)}}&=&\sqrt{E_{n,m,1}} \cap \sqrt{X_{n,m} \cap Y_{m,1} \cap Z_{n,1}} \quad \text{for $2 \leq n \leq 4$, $2 \leq m \leq 4$}, \nonumber \\
\label{equation33}
\sqrt{I_{L_3(n,m,r)}}&=&\sqrt{E_{n,m,r}} \cap \sqrt{X_{n,m} \cap Y_{m,r} \cap Z_{n,r}} \quad \text{for $2 \leq n , m , r \leq 3$}. \nonumber
\end{eqnarray}

Then, we have
\begin{eqnarray}\label{equation11}
\sqrt{I_{L_3(n,1,1)}}&=&\sqrt{E_{n,1,1}} \cap X_{n,1} \cap Y_{1,1} \cap Z_{n,1} \quad \text{for $n=2,3,\cdots,10$}, \\
\label{equation22}
\sqrt{I_{L_3(n,1,1)}}&=&\sqrt{E_{n,m,1}} \cap X_{n,m} \cap Y_{m,1} \cap Z_{n,1} \quad \text{for $2 \leq n \leq 4,2 \leq m \leq 4$}, \\
\label{equation33}
\sqrt{I_{L_3(n,m,r)}}&=&\sqrt{E_{n,m,r}} \cap X_{n,m} \cap Y_{m,r} \cap Z_{n,r} \quad \text{for $2 \leq n , m , r \leq 3$}.
\end{eqnarray}
In fact, It is clear from the following lemma.

\begin{lem}\label{prime ideal}\it
Let $\{ i_1,\cdots,i_s \}$ be a subset of $L$, where $i_1<i_2<\cdots<i_s$. Let $I$ be the ideal $(i_1,\cdots,i_s)$. Then, $I$ is prime ideal.
\end{lem}

\begin{proof}Let $a$ and $b$ be the elemetns of $K[L]$ such that $ab$ belongs to $I$. Suppose neither a nor b belongs to $I$. Then, $a$ and $b$ belong to the polynomial ring over $K$ whose variables are the elements of $L \setminus \{ i_1 , \cdots , i_s \}$. Therefore, since $a$ and $b$ do not contain the variables $i_1, i_2, \cdots, i_s$, it contradicts that $ab$ belongs to $I$. Hence, $a$ is in $I$ or $b$ is in $I$. Therefore, $I$ is prime ideal.
\end{proof}

Now, by using Risa/Asir \cite{5} , we computed the prime decomposition of $\sqrt{E_{n,m,r}}$ appearing in the right-hand side of $(\ref{equation11})$ , $(\ref{equation22})$ and $(\ref{equation33})$. It is as follows:
\begin{eqnarray}\nonumber
\sqrt{E_{n,m,r}}=
\begin{cases}
E_{n,1,1}, &  n=2,3,\cdots,10, \\
E_{n,m,1}, & 2 \leq n \leq 4,2 \leq m \leq 4, \\
E_{n,m,r}, & 2 \leq n,m,r \leq 3.
\end{cases}
\end{eqnarray}

Hence, it follows from $(\ref{equation11})$ , $(\ref{equation22})$ and $(\ref{equation33})$ that we have
\begin{eqnarray}
\sqrt{I_{L_3(n,1,1)}}&=&E_{n,1,1} \cap X_{n,1} \cap Y_{1,1} \cap Z_{n,1} \quad \text{for $n=2,3,\cdots,10$}, \nonumber \\
\sqrt{I_{L_3(n,1,1)}}&=&E_{n,m,1} \cap X_{n,m} \cap Y_{m,1} \cap Z_{n,1} \quad \text{for $2 \leq n \leq 4,2 \leq m \leq 4$}, \nonumber \\
\sqrt{I_{L_3(n,m,r)}}&=&E_{n,m,r} \cap X_{n,m} \cap Y_{m,r} \cap Z_{n,r} \quad \text{for $2 \leq n , m , r \leq 3$}. \nonumber
\end{eqnarray}
Therefore, we proved Theorem $\ref{theo2}$.

\section{Crystal conjecture}

In this section, we introduce the conjectures that occur naturally by Theorem \ref{theo1} and $\ref{theo2}$. It is as follows:

\begin{conj}\label{Crystal conjecture}[Crystal conjecture]\it
The join-meet ideal $I_{L_k(n_1,n_2,\cdots,n_k)}$ is radical.
\end{conj}

We consider that Conjecture \ref{Crystal conjecture} is positive. The reason is as follows. By the proof of theorem \ref{theo1}, it was confirmed the existence of monomial order $\prec$ which is satisfying $\initial_{\prec}(L_2(n_1,n_2))=\sqrt{ \initial_{\prec}(L_2(n_1,n_2))}$. Furthermore, the method of constructing $\prec$ was simple. Hence, for $k=3$, we can conjecture that there may be such a monomial order. Therefore, the following conjecture naturally occurs:

\begin{conj}\label{k=3}\it
For  $k \geq 3$ and $(n_1,\cdots,n_k) \neq (1,\cdots,1)$, it exists a monomial order $\prec'$ such that $I_{L_k(n_1,\cdots,n_k)}$ is squarefree with respect to $\prec'$.
\end{conj}

\begin{remark}[Reason for imposing $(n_1,\cdots,n_k) \neq (1,1,1)$] By [1, Theorem 6.10 (Dedekind)], since we have $L_3(1,1,1)=M_5$, a finite lattice $L_k(1,1,1,\cdots,1)$ is non-distributive modular lattice. Hence, by [2, Theorem 1.3], it do not exist a monomial order such that $I_{L_k(1,\cdots,1)}$ is squarefree. From such a fact, it  imposes $(n_1,\cdots,n_k) \neq (1,1,1)$ in Conjecture \ref{k=3}.
\end{remark}

In Conjecture $\ref{k=3}$, if the method of constructing $\prec'$ can be formulated as an algorithm that does not depend on $k$, Conjecture \ref{Crystal conjecture} be resolved. Hence, it is worth working on Conjecture $\ref{k=3}$. However, in the case $k=3$, although we computed a lot with Risa/Asir \cite{5}, we are not yet checked  a monomial order $\prec'$ which is satisfying $\initial_{\prec'}(L_3(n_1,n_2,n_3))=\sqrt{ \initial_{\prec'}(L_3(n_1,n_2,n_3))}$ for $(n_1,n_2,n_3) \neq (1,1,1)$. From this calculation experiment, unfortunately, Conjecture $\ref{k=3}$ may be negative. On the other hand, we can consider positively that it is very important result in terms of squarefree of join-meet ideal.

\section{Invariance of radicality by adding new relationship}

In this section, at first, we introduce a new finite lattice $L_2(n_1,n_2)[k',i_1,i_2]$ and $O_{n_1}$ which is created by adding a new relationship to $L_2(n_1,n_2)$. Next, we prove Theorem \ref{theo4} and \ref{theo3}. This result is similar to [2, Theorem 3.3] in terms of the opposite approach. and means invariance of radicality by adding new relationship. Hereafter, we explain each $ L_2 (n_1, n_2) [k', i_1, i_2] $ and $ O_ {n_1} $ separately.

\subsection{A finite lattice $L_2(n_1,n_2)[k',i_1,i_2]$}

Let $n_1 \geq 5$ and $n_2 \geq 5$. Let $i_1,i_2$ be non-negative integer and let $i_1 >1$, $4 < i_2 < n_2$ and $i_2-i_1 \geq 2$. Let $k'$ be non-negative integer which satisfies $3 \leq k' \leq n_ 2 - 2$ and $k' \neq n_1,n_2$. We denote $L_2(n_1,n_2)[k',i_1,i_2]$ by $L_2(n_1,n_2)$ which satisfies 
\begin{eqnarray}
a_{i_1} < b_{k'}, \quad b_{k'} < a_{i_2}. \nonumber
\end{eqnarray}
By [1, Theorem 6.10 (Dedekind)], note that $L_2(n_1,n_2)[k',i_1,i_2]$ be a non-modular lattice. In fact, since $b_1 < b_2 < b_3 < \cdots < b_{k'}$ and since $a_1 , a_2 , \cdots , a_{i_1-1}$ are incomparable to $b_1, b_2 , b_3 , \cdots,b_{k'}$ respectively, it exists a sublattice $\{s, a_1, b_1, b_2, b_{k'} \}$ of $L_2(n_1,n_2)[k',i_1,i_2]$ is isomorphic to the pentagon lattice $N_5$.

Before introducing the lemma to prove Theorem $\ref{theo4}$, we need to introduce some notation. Let
\begin{eqnarray}
G_1(i_1,{k'})&=&\{ a_i b_j - s b_{k'} \mid 1 \leq i \leq i_1 , 1 \leq j \leq {k'}-1 \}, \nonumber \\
G_2(i_1,i_2)&=&\{ a_i b_{k'} - a_{i_1} a_{i_2} \mid i_1+1 \leq i \leq i_2-1 \}, \nonumber \\
G_3(i_2,{k'})&=&\{ a_i b_j - b_{k'} t \mid i_2 \leq i \leq n_1 , {k'}+1 \leq j \leq n_2 \}, \nonumber \\
A_1(i_1)&=&\{ a_i s b_{k'} -a_1 s b_{k'} \mid 2 \leq i \leq i_1 \}, \nonumber \\
A_2(i_1,i_2)&=&\{ a_i a_{i_1} a_{i_2} - a_{i_1+1} a_{i_1} a_{i_2} \mid i_1 +2 \leq i \leq i_2 -1 \}, \nonumber \\
A_3(i_2)&=&\{ a_i b_{k'} t -a_{i_2} b_{k'} t \mid i_2+1 \leq i \leq n_1 \}, \nonumber \\
B_1({k'})&=&\{ b_i s b_{k'} -b_1 s b_{k'} \mid 2 \leq i \leq {k'}-1 \}, \quad B_2({k'})=\{ b_i b_{k'} t - b_{k+1} b_{k'} t  \mid {k'}+2 \leq i \leq n_2 \} \nonumber
\end{eqnarray}
and
\begin{eqnarray}
L_{i_1,k'}&=&\{ s , a_1 , \cdots , a_{i_1} , b_1 , \cdots, b_{{k'}-1} , b_{k'} \}, \nonumber \\
L_{i_1,i_2,k'}&=&\{ a_{i_1} , a_{i_1+1} , \cdots , a_{i_2-1} , a_{i_2} , b_{k'} \}, \nonumber \\
L_{i_2,k'}&=&\{ a_{i_2} ,  \cdots , a_{n_1} , b_{k'} , b_{{k'}+1} , \cdots, b_{n_2} , t \}. \nonumber
\end{eqnarray}
Note that a system of generators of $I_{L_2(n_1,n_2)[{k'},i_1,i_2]}$ be $G_1(i_1,{k'}) \cup G_2(i_1,i_2) \cup G_3(i_2,{k'})$.

\begin{lem}\label{theo444}\it
For $i_2-i_1>2$, the set 
\begin{eqnarray}
G_1(i_1,{k'}) \cup G_2(i_1,i_2) \cup G_3(i_2,{k'}) \cup A_1(i_1) \cup A_2(i_1,i_2) \cup A_3(i_2) \cup B_1({k'}) \cup B_2({k'}) \nonumber
\end{eqnarray}
is a Gr\"{o}bner basis of $I_{L_2(n_1,n_2)[k',i_1,i_2]}$ with respect to the inverse lexicographic order induced by $(\ref{prime order})$.
\end{lem}

\begin{proof}
By Theorem \ref{theo1}, $G_1(i_1,{k'}) \cup A_1(i_1) \cup B_1({k'})$ is a Gr\"{o}bner basis with respect to $\prec$ of $L_{i_1,k'}$  and $G_2(i_1,i_2) \cup A_1(i_1,i_2)$ is a Gr\"{o}bner basis with respect to $\prec$ of $L_{i_1,i_2,k'}$. Also, it follows from Theorem \ref{theo1} that $G_3(i_2,{k'}) \cup A_3(i_2) \cup B_2({k'})$ is a Gr\"{o}bner basis with respect to $\prec$ of $L_{i_2,k'}$. Hence, it follows from [1, Lemma 1.27] that the $S$-polynomials which we only have to check are
\begin{eqnarray}
\label{nihon1}
&&S(a_i s b_{k'} -a_1 s b_{k'} , a_j b_{k'} t -a_{i_2} b_{k'} t) \quad \text{for $2 \leq i \leq i_1,i_2+1 \leq j \leq n_1$}, \\
\label{nihon3}
&&S(a_i s b_{k'} -a_1 s b_{k'} , b_j b_k t - b_{{k'}+1} b_{k'} t) \quad \text{for $2 \leq i \leq i_1, {k'}+2 \leq j \leq n_2$}, \\
\label{nihon4}
&&S( a_i b_{k'} t -a_{i_2} b_{k'} t , b_j s b_{k'} -b_1 s b_{k'}) \quad \text{for $i_2+1 \leq i \leq n_1 , 2 \leq j \leq {k'}-1$}, \\
\label{nihon5}
&&S(b_i s b_{k'} -b_1 s b_{k'} , b_j b_{k'} t - b_{{k'}+1} b_{k'} t) \quad \text{for $2 \leq i \leq {k'}-1 , {k'}+2 \leq j \leq n_2$}.
\end{eqnarray}
Now, the result of computation of $(\ref{nihon1}), \cdots , (\ref{nihon5})$  is as follows:
\begin{eqnarray}
S(a_i s b_{k'} - a_1 s b_{k'} , a_j b_{k'} t -a_{i_2} b_{k'} t)
&=&-s a_1 a_j b_{k'} t +s a_i a_{i_2} b_{k'} t \nonumber \\
&=&a_{i_2} t ( a_i s b_{k'} - a_1 s b_{k'} ) - s a_1 (a_j b_{k'} t - a_{i_2} b_{k'} t) \quad \text{for $2 \leq i \leq i_1,i_2+1 \leq j \leq n_1$}, \nonumber \\
S(a_i s b_{k'} -a_1 s b_{k'} , b_j b_{k'} t - b_{{k'}+1} b_{k'} t)
&=&- s a_1 b_j b_{k'} t + s a_i b_{{k'}+1} b_{k'} t \nonumber \\
&=&b_{{k'}+1} t (a_i s b_{k'} - a_1 s b_{k'}) - s a_1 (b_j b_{k'} t - b_{{k'}+1} b_{k'} t) \quad \text{for $2 \leq i \leq i_1, {k'}+2 \leq j \leq n_2$}, \nonumber \\
S( a_i b_{k'} t - a_{i_2} b_{k'} t , b_j s b_{k'} - b_1 s b_{k'})
&=&- s a_{i_2} b_j b_{k'} t + s a_i b_1 b_{k'} t \nonumber \\
&=&s b_1 (a_i b_{k'} t - a_{i_2} b_{k'} t) - a_{i_2} t (b_j s b_{k'} - b_1 s b_{k'}) \quad \text{for $i_2+1 \leq i \leq n_1 , 2 \leq j \leq {k'}-1$}, \nonumber \\
S(b_i s b_{k'} -b_1 s b_{k'} , b_j b_{k'} t - b_{{k'}+1} b_{k'} t)
&=&- s b_1 b_j b_{k'} t + s b_i b_{k'} b_{{k'}+1} t \nonumber \\
&=&b_{k+1} t (b_i s b_{k'} - b_1 s b_{k'}) - s b_1 (b_j b_{k'} t - b_{{k'}+1} b_{k'} t) \quad \text{for $2 \leq i \leq {k'}-1 , {k'}+2 \leq j \leq n_2$}. \nonumber
\end{eqnarray}
Thus, $S$-polynomials $(\ref{nihon1}), \cdots , (\ref{nihon5})$ reduce to $0$. Hence, for $i_2-i_1>2$, the set $G_1(i_1,{k'}) \cup G_2(i_1,i_2) \cup G_3(i_2,{k'}) \cup A_1(i_1) \cup A_2(i_1,i_2) \cup A_3(i_2) \cup B_1({k'}) \cup B_2({k'})$ is a Gr\"{o}bner basis of $I_{L_2(n_1,n_2)[{k'},i_1,i_2]}$ with respect to $\prec'$.
\end{proof}

\begin{lem}\label{theo4441}\it
For $i_2-i_1=2$, the set 
\begin{eqnarray}
G_1(i_1,{k'}) \cup G_2(i_1,i_1+2) \cup G_3(i_1+2,{k'}) \cup A_1(i_1) \cup A_3(i_1+2) \cup B_1({k'}) \cup B_2({k'}) \nonumber
\end{eqnarray}
is a Gr\"{o}bner basis of $I_{L_2(n_1,n_2)[{k'},i_1,i_1+2]}$ with respect to the inverse lexicographic order induced by $(\ref{prime order})$.
\end{lem}

\begin{proof}Since $i_2=i_1+2$, we have $G_2(i_1,i_1+2)=\{ a_{i_1+1} b_{k'} - a_{i_1} a_{i_1+2} \}$. Hence, we have Theorem \ref{theo4441} by  [1, Lemma 1.27] and  the computational result of $(\ref{nihon1}), \cdots , (\ref{nihon5})$.
\end{proof}

Now, we prove Theorem \ref{theo4}.

\begin{proof}[\textup{The proof of Theorem \ref{theo4}}]By Theorem \ref{theo444} and \ref{theo4441}, the join-meet ideal $I_{L_2(n_1,n_2)[{k'},i_1,i_2]}$ is squarefree with respect to the inverse lexicographic order induced by $(\ref{prime order})$. Hence, it is radical.
\end{proof}

\subsection{A finite lattice $O_{n_1}$}

Let $n=n_1$. Let $O_{n}$ be $L_2(n,n)$ which satisfies 
\begin{eqnarray}
a_{i} < b_{i+1} < a_{i+2}, \nonumber
\end{eqnarray}
where $i$ is odd number.

Let prove Therorem \ref{theo3}. The outline of proof of it is to show that a system of generators of the join-meet ideal $I_{O_n}$ is a Gr\"{o}bner basis of $I_{O_n}$ with respect to the inverse lexicographic order induced by $(\ref{prime order})$. In short, we show that $O_{n}$ is a distributive lattice. Note that $\prec$ is a rank reverse lexicographic order on $K[O_n]$; see [1, Example 6.16] for definitions.

\begin{proof}[\textup{The proof of Therorem $\ref{theo3}$}]
At first, we clarify a system of generators of the join-meet ideal $I_{O_n}$. Let $R$ be a system of generators of $I_{O_n}$. Since $a_1,\cdots,a_n$ are incomparable with $b_1,\cdots,b_n$ respectively, then we have
\begin{eqnarray}
R&=&\{ f_{a,b} \mid \text{$a,b \in L$ such that $a$ and $b$ are incomparable} \} \nonumber \\
&=&\{ f_{a_1,b_j} \mid 1 \leq j \leq n \} \cup \cdots \cup \{ f_{a_n,b_j} \mid 1 \leq j \leq n \} \nonumber \\
&=&\bigcup_{i=1}^{n} \{ f_{a_i,b_j} \mid 1 \leq j \leq n \}. \nonumber
\end{eqnarray}
Hence, we must consider the following cases:
\begin{eqnarray}
\text{Case $1$}&:&\text{The calculation of $f_{a_i,b_j}$ for  $i \equiv 0 (\bmod 2)$}, \nonumber \\
\text{Case $2$}&:&\text{The calculation of $f_{a_i,b_j}$ for $i \equiv 1 (\bmod 2)$}. \nonumber
\end{eqnarray}

(\text{Case $1$}) Let $i$ be even nummber. Now, a finite lattice $O_n$ satisfies the following inequality:
\begin{eqnarray}\label{bunpaihannteisnohanasi1}
s \leq b_{1} \leq b_2 \leq \cdots \leq b_{i-2} \leq  a_{i-1} \leq a_i \leq a_{i+1} \leq b_{i+2} \leq \cdots \leq b_{n} \leq t.
\end{eqnarray}
By $(\ref{bunpaihannteisnohanasi1})$, we have
\begin{eqnarray}\label{bunpaihannteisnohanasi2}
b_{1} \leq \cdots \leq b_{i-2} \leq a_i \leq b_{i+2} \leq \cdots \leq b_{n}.
\end{eqnarray}
Hence, it follows from $(\ref{bunpaihannteisnohanasi2})$ that we have $f_{a_i,b_j}=0$ for $j \leq i-2$, $i+2 \leq j$. On the other hand, since $a_i$ is incomparable to $b_{i-1}$, $b_i$, $b_{i+1}$ respectively, we must consider the calculation of $a_i \join b_{\ell}$, $a_i \meet b_{\ell}$, where $\ell=i-1,i,i+1$.\\

First, in the case $j=i-1$, since
\begin{eqnarray}
b_{i-2} \leq a_{i-1} \leq a_i \leq a_{i+1}, \quad b_{i-2} \leq b_{i-1} \leq b_{i} \leq a_{i+1}, \nonumber
\end{eqnarray}
we have $a_i \join b_{i-1} = a_{i+1}$, $a_i \meet b_{i-1} = b_{i-2}$. Hence, for $j=i-1$, we have
\begin{eqnarray}
f_{a_i,b_j}=f_{a_i,b_{i-1}}=a_i b_{i-1} - a_{i+1} b_{i-2}.
\nonumber
\end{eqnarray}

Second, in the case $j=i$, since 
\begin{eqnarray}
a_{i-1} \leq a_i \leq a_{i+1}, \quad a_{i-1} \leq b_{i} \leq a_{i+1}, \nonumber
\end{eqnarray}
we have $a_i \join b_i = a_{i+1}$, $a_i \meet b_i = a_{i-1}$. Hence, for $j=i$, we have
\begin{eqnarray}
f_{a_i,b_j}=f_{a_i,b_i}=a_i b_i - a_{i-1} a_{i+1}.
\nonumber
\end{eqnarray}

Finally, in the case $j=i+1$, since
\begin{eqnarray}
a_{i-1} \leq a_i \leq a_{i+1} \leq b_{i+2}, \quad a_{i-1} \leq b_i \leq b_{i+1} \leq b_{i+2},  \nonumber
\end{eqnarray}
we have $a_i \join b_{i+1} = b_{i+2}$, $a_i \meet b_{i+1} = a_{i-1}$. Hence, for $j=i+1$, we have
\begin{eqnarray}
f_{a_i,b_j}=f_{a_i,b_{i+1}}=a_i b_{i+1} - a_{i-1} b_{i+2}. \nonumber
\end{eqnarray}

Therefore, the polynomial $f_{a_i,b_j}$ is as follows:
\begin{eqnarray}\label{even poly}
f_{a_i,b_j}=
\begin{cases}
0, & j \leq i-2, \\
a_i b_{i-1} - a_{i+1} b_{i-2}, & j=i-1, \\
a_i b_i - a_{i-1} a_{i+1}, & j=i, \\
a_i b_{i+1} - a_{i-1} b_{i+2}, & j=i+1, \\
0, & j \geq i+2.
\end{cases}
\end{eqnarray}
\\

(\text{Case $2$}) Let $i$ be odd nummber.Now, $O_n$ satisfies the following inequality:
\begin{eqnarray}\label{bunpaihannteisnohanasi3}
b_{i-1} \leq a_i \leq b_{i+1}.
\end{eqnarray}
By $(\ref{bunpaihannteisnohanasi3})$, we have
\begin{eqnarray}\label{bunpaihannteisnohanasi4}
s \leq b_{1} \leq \cdots \leq b_{i-1} \leq a_i \leq b_{i+1} \leq \cdots b_{n} \leq t.
\end{eqnarray}
Hence, for $i \neq j$, we have $f_{a_i,b_j}=0$ by $(\ref{bunpaihannteisnohanasi4})$. On the other hand, for $i=j$, since $a_i$ and $b_i$ are incomparable, we have $f_{a_i,b_i}=a_i b_i - b_{i-1} b_{i+1}$ by $(\ref{bunpaihannteisnohanasi3})$. Hence, the polynomial $f_{a_i,b_j}$ is as follows:
\begin{eqnarray}\label{odd poly}
f_{a_i,b_j}=
\begin{cases}
a_i b_i - b_{i-1} b_{i+1}, & i=j, \\
0, & i \neq j.
\end{cases}
\end{eqnarray}

Therefore,  $R$ consists of $(\ref{even poly})$ and $(\ref{odd poly})$.\\

Next, we show that $R$ is Gr\"{o}bner basis of $I_{O_n}$ with respect to compatible monomial order $\prec$. Let $j$ and $r$ be non-negative integer. Let check that $S$-polynomials
\begin{eqnarray}
&S(f_{a_i,b_j},f_{a_k,b_{\ell}})& \text{for $i , k \equiv 0 (\bmod 2)$}, \nonumber \\
&S(f_{a_i,b_j},f_{a_k,b_{\ell}})& \text{for $i , k \equiv 1 (\bmod 2)$}, \nonumber \\
&S(f_{a_i,b_j},f_{a_k,b_{\ell}})& \text{for $i \equiv 1 (\bmod 2) , k \equiv 0 (\bmod 2)$} \nonumber
\end{eqnarray}
reduce to $0$ with respect to generators of $R$.\\

First, we check that $S(f_{a_i,b_j},f_{a_k,b_{\ell}})$ reduces to $0$ with respect to generators of $R$, where $i, k \equiv 0 \pmod 2$. It follows from $(\ref{even poly})$ that we have
\begin{eqnarray}
\initial_{\prec}(f_{a_i,b_j})=
\begin{cases}
0, & j \leq i-2, \\
a_i b_{i-1}, & j=i-1, \\
a_i b_i, & j=i, \\
a_i b_{i+1}, & j=i+1, \\
0, & j \geq i+2.
\end{cases}
\nonumber
\end{eqnarray}
From above equation, for each $i=k$ and $i \neq k$, it is necessary to consider the calculation of $S(f_{a_i,b_j},f_{a_k,b_{\ell}})$.\\

$(\text{The case $i=k$})$ Since a initial monomial of $f_{a_i,b_j}$ and $f_{a_i,b_{\ell}}$ are as follows:
\begin{eqnarray}
\initial_{\prec}(f_{a_i,b_j})=
\begin{cases}
0, & j \leq i-2, \\
a_i b_{i-1}, & j=i-1, \\
a_i b_i, & j=i, \\
a_i b_{i+1}, & j=i+1, \\
0, & j \geq i+2,
\end{cases} \quad
\initial_{\prec}(f_{a_i,b_{\ell}})=
\begin{cases}
0, & \ell \leq i-2, \\
a_i b_{i-1}, & \ell=i-1, \\
a_i b_i, & \ell=i, \\
a_i b_{i+1}, & \ell=i+1, \\
0, & \ell \geq i+2.
\end{cases}
\nonumber
\end{eqnarray}
Thus, for $j=\ell=i-1,i,i+1$, we have $S(f_{a_i,b_j},f_{a_i,b_{\ell}})=0$. Hence, we only have to check out that $S$-polynomial
\begin{eqnarray}
S(f_{a_i,b_j},f_{a_i,b_{\ell}})=
\begin{cases}\label{popi}
S(f_{a_i,b_{i-1}},f_{a_i,b_i}), & (j,\ell)=(i-1,i), \\
S(f_{a_i,b_{i-1}},f_{a_i,b_{i+1}}), & (j,\ell)=(i-1,i+1), \\ 
S(f_{a_i,b_i},f_{a_i,b_{i+1}}), & (j,\ell)=(i,i+1).
\end{cases}
\end{eqnarray}
reduces to $0$. Now, it follows from $(\ref{popi})$ that we have
\begin{eqnarray}
S(f_{a_i,b_j},f_{a_i,b_{\ell}})=
\begin{cases}
a_{i+1} f_{a_{i-1},b_{i-1}}, & (j,\ell)=(i-1,i), \\ 
b_{i+2} f_{a_{i-1},b_{i-1}} - b_{i-2} f_{a_{i+1},b_{i+1}}, & (j,\ell)=(i-1,i+1), \\ 
- a_{i-1} f_{a_{i+1},b_{i+1}}, & (j,\ell)=(i,i+1).
\end{cases}
\nonumber
\end{eqnarray}
Hence, $S(f_{a_i,b_j},f_{a_i,b_{\ell}})$ reduces to $0$.\\

$(\text{The case $i \neq k$})$ Suppose $i < k$ and let $\varepsilon=k-i$. Then, a initial monomial of $f_{a_i,b_j}$ and  $f_{a_k,b_{\ell}}$ are as follows:
\begin{eqnarray}\label{8888888888888}
\initial_{\prec}(f_{a_i,b_j})=
\begin{cases}
0, & j \leq i-2, \\
a_i b_{i-1}, & j=i-1, \\
a_i b_i, & j=i, \\
a_i b_{i+1}, & j=i+1, \\
0, & j \geq i+2,
\end{cases} \quad
\initial_{\prec}(f_{a_k,b_{\ell}})=
\begin{cases}
0, & \ell \leq i +  \varepsilon -2, \\
a_{i+\varepsilon} b_{i+ \varepsilon-1}, & \ell=i+\varepsilon-1, \\
a_{i+\varepsilon} b_{i+ \varepsilon}, & \ell=i+ \varepsilon, \\
a_{i+\varepsilon} b_{i+ \varepsilon+1}, & \ell=i+\varepsilon+1, \\
0, & \ell \geq i+ \varepsilon+2.
\end{cases}
\nonumber
\end{eqnarray}
Then, we must consider that we look for $\varepsilon$ such that $j$ and $\ell$ satisfy $b_j=b_{\ell}$ in the above calculation result. Below, we consider the following cases:
\begin{eqnarray}
&\text{Case $2.1$}:&b_{i-1}=b_{i+ \varepsilon-1}, b_{i+ \varepsilon} , b_{i+ \varepsilon+1}, \nonumber \\
&\text{Case $2.2$}:&b_{i}=b_{i+ \varepsilon-1}, b_{i+ \varepsilon} , b_{i+ \varepsilon+1}, \nonumber \\
&\text{Case $2.3$}:&b_{i+1}=b_{i+ \varepsilon-1}, b_{i+ \varepsilon} , b_{i+ \varepsilon+1}. \nonumber
\end{eqnarray}

(\text{Case $2.1$}) In this case, since
\begin{eqnarray}
i-1=i+ \varepsilon-1, i+ \varepsilon, i+ \varepsilon+1, \nonumber
\end{eqnarray}
we have $\varepsilon=0,-1,-2$. Since $\varepsilon > 0$, then $\varepsilon=0,-1,-2$ can't satisfy Case $2.1$.\\

(\text{Case $2.2$}) In this case, since
\begin{eqnarray}\label{varepsil2}
i=i+ \varepsilon-1, i+ \varepsilon, i+ \varepsilon+1, \nonumber
\end{eqnarray}
we have $\varepsilon=-1,0,1$. Since $\varepsilon > 0$ and since it is even, then $\varepsilon=-1,0,1$ can't satisfy Case $2.2$.\\

(\text{Case $2.3$}) In this case, since
\begin{eqnarray}
i+1=i+ \varepsilon-1, i+ \varepsilon, i+ \varepsilon+1, \nonumber
\end{eqnarray}
we have $\varepsilon=0,1,2$. Since $\varepsilon > 0$ and since it is even, then $\varepsilon=2$ only satisfies Case $2.3$.\\

From three cases, at first, for $\varepsilon>2$, since $\initial_{\prec}(f_{a_i,b_j})$ and $\initial_{\prec}(f_{a_k,b_{\ell}})$ are relatively prime, $S(f_{a_i,b_j},f_{a_k,b_{\ell}})$ reducess to respect to $f_{a_i,b_j},f_{a_k,b_{\ell}}$. Next, for $\varepsilon=2$, we have
\begin{eqnarray}
\initial_{\prec}(f_{a_i,b_j})=
\begin{cases}
0, & j \leq i-2, \\
a_i b_{i-1}, & j=i-1, \\
a_i b_i, & j=i, \\
a_i b_{i+1}, & j=i+1, \\
0, & j \geq i+2,
\end{cases} \quad
\initial_{\prec}(f_{a_k,b_{\ell}})=
\begin{cases}
0, & \ell \leq i, \\
a_{i+2} b_{i+1}, & \ell=i+1, \\
a_{i+2} b_{i+2}, & \ell=i+2, \\
a_{i+2} b_{i+3}, & \ell=i+3, \\
0, & \ell \geq i+4.
\end{cases}
\nonumber
\end{eqnarray}
Hence, for 
\begin{eqnarray}
(j,\ell)=(i-1,i+1),(i-1,i+2),(i-1,i+3),(i,i+1),(i,i+2),(i,i+3),(i+1,i+2),(i+1,i+3), \nonumber
\end{eqnarray}
since $\initial_{\prec}(f_{a_i,b_j})$ and $\initial_{\prec}(f_{a_k,b_{\ell}})$ are relatively prime, $S(f_{a_i,b_j}, f_{a_k,b_{\ell}})$ reduces to $0$ with respect to $f_{a_i,b_j},f_{a_k,b_{\ell}}$. On the other hand, for $(j,\ell)=(i+1,i+1)$, since we have
\begin{eqnarray}
S(f_{a_i,b_{i+1}},f_{a_{i+2},b_{i+1}})=a_{i+3} f_{a_i,b_i} - a_{i-1} f_{a_{i+2},b_{i+2}},
\nonumber
\end{eqnarray}
$S(f_{a_i,b_{i+1}},f_{a_{i+2},b_{i+1}})$ reduces to $0$ with respect to  $f_{a_i,b_i}$, $f_{a_{i+2},b_{i+2}}$. Therefore, we checked that $S(f_{a_i,b_j},f_{a_k,b_{\ell}})$ reduces to $0$ with respect to generators of $R$.\\

Second, we check that $S(f_{a_i,b_j},f_{a_k,b_{\ell}})$ reduces to $0$ with respect to generators of $R$, where $i, k \equiv 1 \pmod 2$. For $i \neq j$ or $k \neq \ell$, it follows from $(\ref{odd poly})$ that  $S(f_{a_i,b_j},f_{a_k,b_{\ell}})$ reduces to $0$ with respect to $f_{a_i,b_j},f_{a_k,b_{\ell}}$. On the other hand, for $i=j$ and $k=\ell$, since $\initial_{\prec}(f_{a_i,b_i})=a_i b_i$ and $\initial_{\prec}(f_{a_k,b_k})=a_k b_k$, we have to consider the calculation of $S(f_{a_i,b_i},f_{a_k,b_k})$ for each $i=k$ and $i \neq k$.\\

$(\text{The case $i=k$})$ It follows from $i=k$ that we have $S(f_{a_i,b_i},f_{a_k,b_k})=0$. \\

$(\text{The case $i \neq k$})$ It follow from $i \neq k$ that $\initial_{\prec}(f_{a_i,b_i})$ and $\initial_{\prec}(f_{a_k,b_k})$ are relatively prime.  Hence, $S(f_{a_i,b_i},f_{a_k,b_k})$ reduces to $0$ with respect to $f_{a_i,b_i}$, $f_{a_k,b_k}$.\\ 

Therefore,  we checked that $S(f_{a_i,b_j},f_{a_k,b_{\ell}})$ reduces to $0$ with respect to generators of $R$.\\

Finally, we check that $S(f_{a_i,b_j},f_{a_k,b_{\ell}})$ reduces to $0$ with respect to generators of $R$, where $i \equiv 1 \pmod 2$ and $k \equiv 0 \pmod 2$. By $(\ref{even poly})$ and $(\ref{odd poly})$, we have
\begin{eqnarray}
\initial_{\prec}(f_{a_i,b_j})=
\begin{cases}
a_i b_i, & i=j, \\
0, & i \neq j,
\end{cases}\quad
\initial_{\prec}(f_{a_k,b_{\ell}})=
\begin{cases}
0, & \ell \leq k-2, \\
a_k b_{k-1}, & \ell=k-1, \\
a_k b_k, & \ell=k, \\
a_k b_{k+1}, & \ell=k+1, \\
0, & \ell \geq k+2.
\end{cases}
\nonumber
\end{eqnarray}
From above result, if $i \neq j$, then we have $S(f_{a_i,b_j},f_{a_k,b_{\ell}})=-f_{a_k,b_{\ell}}$. On the other hand, for $i=j$, we have
\begin{eqnarray}
S(f_{a_i,b_j},f_{a_k,b_{\ell}})=
\begin{cases}
- b_{k-2} f(a_k,b_k), & i=j=\ell=k-1, \\
- b_{k+2} f(a_k,b_k), & i=j=\ell=k+1.
\end{cases}
\nonumber
\end{eqnarray}
Hence,  we checked that $S(f_{a_i,b_j},f_{a_k,b_{\ell}})$ reduces to $0$ with respect to generators of $R$.\\

Therefore, we showed that $R$ is a Gr\"{o}bner basis of $I_{O_n}$ with respect to $\prec$. By [1, Theorem 6.17], $O_n$ is a distributive lattice. Hence, it follows from [1, Theorem 6.21] that $I_{O_n}$ is radical.
\end{proof}

\subsection{Some comments} We end this section with a few comments. By [1, Theorem 6.10 (Dedekind)], we can show that $O_{n}$ is a distributive lattice by using the Hasse diagram of its. However, it may be surprisingly difficult to write the process of proof in detail. Rather than that, we can consider that writing the process of proof is very easy by using [1, Theorem 6.17]. In short, we can use mathematical formulas to write the process of proof in detail.

\section{Topics related to special finite lattice}

In this section, we introduce topics related to a distributive lattice $O_ {n_1}$. Note that it has little to do with the gist of this paper.

\subsection{Number-theoretic characterization}

In this subsection, we introduce the relationship between $O_ {n_1}$ and number theory. By Theorem $\ref{theo3}$, a finite lattice $O_{n_1}$ is distributive lattice. On the other hand, it looks abstract as the structure of the set and it has a difficult shape. However, it is not. By the following theorem, we can see $O_{n_1}$ as number-theoretic finite lattice whose shape is very easy.

\begin{theo}\label{EEEEEE}\it Let $p$ and $q$ be prime number with $p \neq q$. For non-negative integer $k$, let 
\begin{eqnarray}
L_{p,q,k}=\bigcup_{r=1}^k C_{p,q,r} \nonumber
\end{eqnarray}
ordered by divisibility, where
\begin{eqnarray}
C_{p,q,1}=\{ 1 , p , q , p^2 , pq , p^2q \}, \quad C_{p,q,r}=\{ p^{r-1} q^r , p^r q^r , p^{r+1} q^r , p^{r+1} q^{r-1} \}, \quad r>1. \nonumber
\end{eqnarray}
Then, $O_{2k}$ is isomorphic to $L_{p,q,k}$.\\
\end{theo}

\begin{proof}At first, we prepare some things necessary to prove Theorem $\ref{EEEEEE}$.\\

Let $L_{p,q,k}=L_k$. Let define the map $h_{1,k}$ : $O_{2k} \to L_k$ by setting
\begin{eqnarray}
&&h_{1,k}(s)=1, \quad h_{1,k}(a_1)=p, \nonumber \\
&&h_{1,k}(a_{2r})=p^{r+1} q^{r-1}, \quad h_{1,k}(b_{2r-1})=p^{r-1} q^r \quad \text{for $r=1,\cdots,k$}, \nonumber \\
&&h_{1,k}(a_{2r+1})=h_{1,k}(a_{2r}) \join h_{1,k}(b_{2r}), \quad h_{1,k}(b_{2r})=h_{1,k}(a_{2r-1}) \join h_{1,k}(b_{2r-1}) \quad \text{for $r=1,\cdots,k$}. \nonumber
\end{eqnarray}
We define the map $h_{2,k}$ : $L_k \to O_{2k}$ by setting
\begin{eqnarray}
&&h_{2,k}(1)=s, \quad h_{2,k}(p)=a_1, \nonumber \\
&&h_{2,k}(p^{r+1} q^{r-1})=a_{2r}, \quad h_{2,k}(p^{r-1} q^r)=b_{2r-1} \quad \text{for $r=1,\cdots,k$}, \nonumber \\
&&h_{2,k}(p^{r+1} q^r)=a_{2r+1}, \quad h_{2,k}(p^r q^r)=b_{2r} \quad \text{for $r=1,\cdots,k$}. \nonumber
\end{eqnarray}

Now, we prove Theorem $\ref{EEEEEE}$. First, we show that $h_{1,k}$ is bijective. Since
\begin{eqnarray}
L_k
&=&\{ 1 , p , q , p^2 , pq , p^2q \} \cup \biggr ( \bigcup_{r=2}^k C_{p,q,r} \biggr ) \nonumber \\
&=&\{ h_{1,k}(s) , h_{1,k}(a_1) , f(b_1) , f(a_2) , f(b_2)=f(a_1) \join f(b_1) , f(a_3)=f(a_2) \join f(b_2) \} \cup \biggr ( \bigcup_{r=2}^k C_{p,q,r} \biggr ), \nonumber
\end{eqnarray}
the mapping $h_{1,k}$ is bijective for $k=1$. On the other hand, in the case $k>1$, since
\begin{eqnarray}
&& C_{p,q,r} \nonumber \\
&=&\{ h_{1,k}(a_{2r}) , h_{1,k}(b_{2r-1}) , h_{1,k}(b_{2r})=h_{1,k}(a_{2r-1}) \join h_{1,k}(b_{2r-1}) , h_{1,k}(a_{2r+1})=h_{1,k}(a_{2r}) \join h_{1,k}(b_{2r}) \} \quad \text{for $r=2,\cdots,k$}, \nonumber
\end{eqnarray}
the set $C_{p,q,r}$ can be described inductively by $\{ h_{1,k}(s) , h_{1,k}(a_1) , h_{1,k}(b_1) , h_{1,k}(a_2) , h_{1,k}(b_2) , h_{1,k}(a_3) \}$ and $C_{p,q,r'}$, where $1 \leq r' \leq r-1$. Hence, for $k>1$, $h_{1,k}$ is bijective. Therefore, $h_{1,k}$ is bijective.\\

Second, we show that $h_{1,k}$ is order-preserving. By definition of $h_{1,k}$, it follows from a subset $\{a_i\}_{i=1}^{2k}$ of $O_{2k}$ that we have
\begin{eqnarray}
h_{1,k}(a_1)=p < p^2=h_{1,k}(a_2), \quad h_{1,k}(a_2)=p^2<h_{1,k}(a_3)=p^2q, \nonumber \\
\cdots , h_{1,k}(a_{2k-1})=p^{k} q^{k-1} < p^{k+1} q^{k-1}=h_{1,k}(a_{2k}). \nonumber
\end{eqnarray}
Hence, $h_{1,k}(a_i) < h_{1,k}(a_{i+1})(i=1,2,\cdots,2k-1)$. Also, it follows from a subset $\{b_i\}_{i=1}^{2k}$ that we get
\begin{eqnarray}
h_{1,k}(b_1)=q < pq=h_{1,k}(b_2), \quad h_{1,k}(b_2)=pq <h_{1,k}(b_3)=p q^2, \nonumber \\
\cdots , h_{1,k}(b_{2k-1})=p^{k-1} q^k < p^k q^k=h_{1,k}(b_{2k}). \nonumber
\end{eqnarray}
Hence, for $i=1,\cdots,2k-1$, we have $h_{1,k}(b_i) < h_{1,k}(b_{i+1})$. On the other hand, $L_k$ satisfies the following inequality:
\begin{eqnarray}
&&h_{1,k}(a_1)=p < pq=h_{1,k}(b_2), \quad h_{1,k}(b_2)=pq < p^2 q=h_{1,k}(a_3), \nonumber \\
&&\cdots , h_{1,k}(a_{r})=p^r q^{r-1} < p^r q^r=h_{1,k}(b_{r+1}), \quad h_{1,k}(b_{r+1})=p^r q^r  < p^{r+1} q^r=h_{1,k}(a_{r+2}), \nonumber \\
&&\cdots, h_{1,k}(a_{2k-1})=p^k q^{k-1} < p^k q^k=h_{1,k}(b_{2k}), \quad h_{1,k}(b_{2k})=p^k q^k < p^{k+1} q^k=h_{1,k}(t) \quad \text{for $r=1,3,\cdots,2k-1$} \nonumber
\end{eqnarray}
Hence, we have $h_{1,k}(a_r) < h_{1,k}(b_{r+1}) < h_{1,k}(a_{r+2})$ for $i=1,3,\cdots,2k-1$. Thus, for $a$ and $b$ belonging to $O_{2k}$ with $a < b$, we have $h_{1,k}(a) < h_{1,k}(b)$. Therefore, we showed that $h_{1,k}$ is order-preserving.\\

Third, we show that $h_{2,k}$ is the inverse mapping of $h_{1,k}$.  Let $h_{21,k}$ be the composite mapping $h_{2,k} \circ h_{1,k}$ and let $h_{12,k}$ the composite mapping $h_{1,k} \circ h_{2,k}$. For an arbitrary element belonging to $O_{2k}$, we have the following result:
\begin{eqnarray}
h_{21,k}(s)&=&h_{2,k}(h_{1,k}(s))=h_{2,k}(1)=s, \nonumber \\
h_{21,k}(a_1)&=&h_{2,k}(h_{1,k}(a_1))=h_{2,k}(p)=a_1, \nonumber \\
h_{21,k}(a_{2r})&=&h_{2,k}(h_{1,k}(a_{2r}))=h_{2,k}(p^{r+1} q^{r-1})=a_{2r} \quad \text{for $r=1,2,\cdots,k$}, \nonumber \\
h_{21,k}(a_{2r+1})&=&h_{2,k}(h_{1,k}(a_{2r+1}))=h_{2,k}(p^{r+1} q^r)=a_{2r+1} \quad \text{for $r=1,2,\cdots,k$}, \nonumber \\
h_{21,k}(b_{2r-1})&=&h_{2,k}(h_{1,k}(b_{2r-1}))=h_{2,k}(p^{r-1} q^r)=b_{2r-1} \quad \text{for $r=1,2,\cdots,k$}, \nonumber \\
h_{21,k}(b_{2r})&=&h_{2,k}(h_{1,k}(b_{2r}))=h_{2,k}(p^r q^r)=b_{2r} \quad \text{for $r=1,2,\cdots,k$}. \nonumber
\end{eqnarray}
Hence, we have $h_{21,k}=\text{id}_{O_{2k}}$. On the other hand, for an arbitrary element belonging to $L_{k}$, we have the following result:
\begin{eqnarray}
h_{12,k}(1)&=&h_{1,k}(h_{2,k}(1))=h_{1,k}(s)=1, \nonumber \\
h_{12,k}(p)&=&h_{1,k}(h_{2,k}(p))=h_{1,k}(a_1)=p, \nonumber \\
h_{12,k}(p^{r+1} q^{r-1})
&=&h_{1,k}(h_{2,k}(p^{r+1} q^{r-1}))=h_{1,k}(a_{2r})=p^{r+1} q^{r-1}, \quad \text{for $r=1,2,\cdots,k$}, \nonumber \\
h_{12,k}(p^{r-1} q^r)
&=&h_{1,k}(h_{2,k}(p^{r-1} q^r))=h_{1,k}(b_{2r-1})=p^{r+1} q^{r-1} \quad \text{for $r=1,2,\cdots,k$}, \nonumber \\
h_{12,k}(p^{r+1} q^r)
&=&h_{1,k}(h_{2,k}(p^{r+1} q^r))=h_{1,k}(a_{2r+1})=p^{r+1} q^r \quad \text{for $r=1,2,\cdots,k$}, \nonumber \\
h_{12,k}(p^r q^r)
&=&h_{1,k}(h_{2,k}(p^r q^r))=h_{1,k}(b_{2r})=p^r q^r \quad \text{for $r=1,2,\cdots,k$}. \nonumber
\end{eqnarray}
Hence, we have $h_{12,k}=\text{id}_{L_k}$. Therefore, $h_{2,k}$ is the inverse mapping of $h_{1,k}$.\\

Finally, we show that $h_{2,k}$ is order-preserving. Since 
\begin{eqnarray}
s < a_1 < a_2 < \cdots < a_{2k} < t , \quad s < b_1 < b_2 < \cdots < b_{2k} < t \nonumber
\end{eqnarray}
in $O_{2k}$, we have
\begin{eqnarray}
&& h_{2,k}(1) < h_{2,k}(p), \nonumber \\
&& h_{2,k}(p^r q^{r-1}) < h_{2,k}(p^{r+1} q^{r-1}) < h_{2,k}(p^{r+1} q^r) \quad \text{for $r=1,2,\cdots,k$}, \nonumber \\
&& h_{2,k}(p^{r-1} q^{r-1}) < h_{2,k}(p^{r-1} q^r) < h_{2,k}(p^r q^r) \quad \text{for $r=1,2,\cdots,k$}. \nonumber
\end{eqnarray}
On the other hand, for $r=1,3,\cdots, 2k-1$, since $a_r < b_{r+1} < a_{r+2}$, we have
\begin{eqnarray}
h_{2,k}(p^r q^{r-1}) < h_{2,k}(p^r q^r) < h_{2,k}(p^{r+1} q^r). \nonumber
\end{eqnarray}
Hence, $h_{2,k}$ is order-preserving. Therefore, $O_{2k}$ is isomorphic to $L_k$.
\end{proof}

\subsection{Gorenstein ring}

In this subsection, we give a non-trivial answer to the question of whether the Hibi ring $R_K[O_{n_1}]=K[O_{n_1}]/I_{O_{n_1}}$ is Gorenstein. Let $n_1=n$. Let $P_{n}$ denote the subposet of $O_{n}$ consisting of all join-irreducible elements of $O_{n}$. By [1, Theorem 6.4 (Birkhoff)], we have $O_{n}=J(P_n)$. Then, we obtain the following theorem.

\begin{theo}\label{6.3}\it For $n \geq 4$, the Hibi ring $R_K[O_{n}]$ is not Gorenstein.
\end{theo}

\begin{proof}Suppose that 
\begin{eqnarray}
P_4=\{ a_1 , a_2 , a_4 , b_1 , b_3 \} \nonumber
\end{eqnarray}
is pure. Then, it follows from [1, Lemma 6.12] that $P_4$ possesses a rank function $\sigma$. Since $a_4$ covers $a_2$ and $a_2$ covers $a_1$ in $P_4$, we have
\begin{eqnarray}\label{p41}
\sigma(a_4)=\sigma(a_2)+1=\sigma(a_1)+2=2.
\end{eqnarray}
On the other hand, since $b_1 < b_2 < a_3 < a_4$ in $O_{4}$, $a_4$ covers $b_1$ in $P_4$. Thus, we have
\begin{eqnarray}\label{p42}
\sigma(a_4)=\sigma(b_1)+1=1.
\end{eqnarray}

Hence, $(\ref{p41})$ and $(\ref{p42})$ contradict the uniqueness of $\sigma$. Therefore, $P_4$ is not pure.

For $n \geq 4$, $P_4$ is a subposet of $P_n$. Since $P_4$ is not pure, $P_n$ is not pure for $n \geq 4$. Hence, it follows from [4, Theorem 1.29] that $R_K[O_{n}]$ is not Gorenstein for $n \geq 4$.
\end{proof}

\end{document}